\documentclass[11pt,letterpaper]{amsart}
\usepackage{amsrefs}
 \usepackage{amsfonts,mathrsfs}
 \usepackage{amsthm}
 \usepackage{amssymb}
 \usepackage{enumerate}
 \usepackage{tikz-cd}
 \usepackage{cleveref}
 \usepackage{enumitem}

\theoremstyle{definition}
\newtheorem{Def}{Definition}[section]

\newtheorem{rem}[Def]{Remark}
\theoremstyle{plain}
\newtheorem{prop}[Def]{Proposition}
\newtheorem{thm}[Def]{Theorem}
\newtheorem*{thm*}{Theorem}
\newtheorem{lem}[Def]{Lemma}

\newtheorem{cor}[Def]{Corollary}
\newtheorem*{cor*}{Corollary}

\newtheorem*{con*}{Conjecture}

\newtheorem*{verm*}{Vermutung}

\newcommand{\nfield}{\mathbf n}

\newcommand{\quash}[1]{}

\usepackage{mathrsfs}

\newcommand{\im}{\operatorname{im}}

\newcommand{\Rea}{\operatorname{Re}}

\newcommand{\T}{\operatorname{T}}

\newcommand{\GL}{\operatorname{{\mathbf GL}}}
\newcommand{\SL}{\operatorname{{\mathbf SL}}}

\newcommand{\cA}{{\mathcal A}}

\newcommand{\C}{{\mathbb C}}
\newcommand{\R}{{\mathbb R}}

\newcommand{\Q}{{\mathbb Q}}
\newcommand{\N}{{\mathbb N}}
\newcommand{\Z}{{\mathbb Z}}

\title[Periodic hypersurfaces and Lee--Yang polynomials]{Periodic hypersurfaces and Lee--Yang polynomials}
\author{Lior Alon}
\address{Massachusetts Institute of Technology, USA} 
\email{lioralon@mit.edu}
\author{Mario Kummer}
\address{Technische Universit\"at, Dresden, Germany} 
\email{mario.kummer@tu-dresden.de}

\begin{document}

\begin{abstract}
We study periodic measures on $\R^n$ whose Fourier transform is 
confined to a proper double cone, in the sense of Meyer’s notion of 
\emph{lighthouse measures}. Lee--Yang polynomials provide a natural family of examples: it follows 
from the work of Kurasov and Sarnak that the torus zero sets of such 
polynomials are hypersurfaces supporting directional lighthouse measures.

We prove a rigidity theorem showing that, under mild assumptions, this 
is essentially the only possibility. Any periodic $C^{1+\varepsilon}$ hypersurface 
supporting a directional lighthouse measure must arise as the torus zero 
set of an essentially Lee--Yang polynomial.
The proof is based on the recent classification of one-dimensional 
Fourier quasicrystals and provides a geometric interpretation of this theory.
\end{abstract}
\maketitle

\section{Introduction}

In Fourier analysis, strong localization of a function in physical space
typically prevents localization of its Fourier transform, due to the uncertainty principle. 
While this phenomenon is well understood for functions, the situation for measures is more subtle: extreme localization of a measure and its Fourier transform can occur simultaneously.
Two fundamental examples illustrate this phenomenon. 
First, the Poisson summation formula shows that the Fourier support of the counting measure on $\Z^n$ is again $\Z^n$. 
Second, if $V \subset \R^n$ is a vector subspace, then the Fourier transform of Lebesgue measure on $V$ is supported on its orthogonal complement $V^\perp$. 
In both cases, the support and the Fourier support are unbounded but of Lebesgue measure zero.

In this work, we consider the $\Z^n$-periodic setting in $\R^n$. 
Given a $\Z^n$-periodic Radon measure $m$, we define its Fourier coefficients by
\[
\widehat{m}(k)=\int_{\R^n/\Z^n} e^{-2\pi i\langle k,x\rangle}\,dm(x),
\qquad k\in\Z^n.
\]
In particular, $\int_{\R^n} \hat{f}\,dm = \sum_{k\in\Z^n} \widehat{m}(k)\,f(k)$ for any Schwartz function $f$ on $\R^n$. 
The examples above admit a natural periodic analogue: if $V \subset \R^n$ is a rational subspace and $m$ is the Lebesgue measure on $V+\Z^n$, then $\widehat{m}$ is supported on $\Z^n \cap V^\perp$. 
Here, the support of $m$, while unbounded, is confined to a thin subset of $\R^n$, and its Fourier support is likewise confined to a thin subset of $\Z^n$.

Beyond such flat examples, it is less obvious whether a measure can remain geometrically thin while its Fourier transform is confined to a small set of directions. 
It turns out that this question is closely related to the study of (not necessarily periodic) discrete measures whose Fourier support is also a discrete measure. 
If such a measure is tempered, it is called a \emph{crystalline measure}. Meyer \cite{Meyer2023} observed that several constructions of higher-dimensional
crystalline measures can be realized as pointwise products of singular measures whose
Fourier transforms are supported in cones. This insight led Meyer to define
lighthouse measures.

\begin{Def}[Lighthouse measure]
A pair $(\mu,C)$ of a non-negative $\Z^n$-periodic Radon measure $\mu$
on $\R^n$ and a proper\footnote{A closed convex cone $C\subset\R^n$ is a \emph{proper cone} if it has non-empty interior and if there is a hyperplane $H\subset\R^n$ with $C\cap H=\{0\}$.} cone $C\subset\R^n$ is called a \emph{lighthouse} if
\begin{enumerate}
    \item $\mu$ is supported on a closed set of Lebesgue measure zero,
    \item $\mathrm{supp}(\widehat{\mu})\subset \Z^n\cap(-C\cup C)$.
\end{enumerate}
\end{Def}

In this work, we investigate the structure of lighthouse measures.
Our first result shows that, under mild assumptions, their support must
have codimension one.

\begin{thm}[Dimension of support]\label{thm: dimension}
Let $(\mu,C)$ be a $\Z^n$-periodic lighthouse measure in $\R^n$. If there
exists a $d$-dimensional $C^{1+\epsilon}$-submanifold $M\subset\R^n$ and
an open set $B$ such that $d\mu|_B=\rho\,d\sigma|_B$ with $\rho>0$
continuous, then $d=n-1$. Moreover, $M$ admits an orientation for which
the normal satisfies $\nfield(x)\in C$ for all $x\in M\cap B$.
\end{thm}
Hence, from here on, we focus on lighthouse measures supported on hypersurfaces.
The condition $\nfield(x)\in C$ has a natural geometric consequence: 
any direction $\ell\in\R^n$ which is \emph{positive} on $C$, i.e.
$\langle \ell, x \rangle > 0 \text{ for all } x \in C \setminus \{0\},$ satisfies
\[
\langle \ell, \nfield(x) \rangle > 0 \quad \text{for all } x \in \Sigma.
\]
Such transversal directions naturally define measures on hypersurfaces that we call \emph{directed measures}. 
Given a direction $\ell\in\mathbb{R}^n$, the \emph{$\ell$-directed measure} $m_\ell$ on $\Sigma$ is the positive Radon measure with density
\[
dm_{\ell}(x)=|\langle \mathbf{n}(x),\ell\rangle|\,d\sigma(x),
\]
where $\mathbf{n}(x)$ is the unit normal to $\Sigma$ and $d\sigma$ is the surface measure\footnote{The $n-1$ dimensional measure inherited from the Lebesgue measure on $\R^n$}. In the presence of the lighthouse condition, the normal vector satisfies 
$\langle \mathbf{n}(x),\ell\rangle>0$ on the support, and hence the absolute value may be removed. In such case, the map $\ell\mapsto m_\ell$ is linear on the set of directions $\ell$ that are positive on the lighthouse cone.

Meyer provided a construction of lighthouse measures in \cite{Meyer2023}, using inner functions, and observed a connection to the crystalline measures that were constructed by Kurasov--Sarnak \cite{kurasovsarnak} using \emph{Lee--Yang polynomials}. The precise statement is as follows, see \cite{alon2025higher}*{Theorem 4.32}. Let $p$ be a Lee--Yang polynomial and let $\Sigma=f^{-1}(0)$ be the zero set of the exponential polynomial
\[
f(x)=p(e^{2\pi i x_{1}},\ldots,e^{2\pi i x_{n}})=:p(\exp(2\pi i x)).
\]
Then for $C=\R_{\ge0}^n$ and any positive direction $\ell\in \R_{+}^n$, the pair $(m_{\ell},C)$ is a lighthouse. 
This construction naturally extends to a wider class of polynomials:
We say that a Laurent polynomial $p$ is \emph{essentially Lee--Yang} if there exists a nonempty open convex cone $K \subset \mathbb{R}^n$ such that
\[
p(\exp(z)) \neq 0 \quad \text{for all \ }\   z \in \mathbb{C}^n \text{ with $\ \Rea (z)$ in}\    K \cup (-K),
\]
where $\exp(z)=(e^{z_1},\ldots,e^{z_n})$ and $\Rea (z)=(\Rea (z_1),\ldots,\Rea (z_n))$. Lee--Yang polynomials are the special case when $K$ is the positive orthant $\mathbb{R}_{+}^n$. If $p$ is essentially Lee-Yang, then, up to possibly taking a smaller $K$, there exists a proper cone $C$ for which $K=\mathrm{int}(C^{*})$ is the cone of positive directions on $C$. 

In \Cref{thm: moreover} we show that, assuming $\Sigma=\{x\in\R^n\mid p(\exp(2\pi ix))=0\}$ to be non-singular, for any positive direction $\ell\in K$ the pair $(m_\ell,C)$ is a lighthouse. 

\medskip

Our main result, \Cref{thm: main}, is an inverse theorem showing that this class precisely captures hypersurfaces supporting lighthouse measures.

\medskip

\textbf{The main result of this paper is that if $\Sigma$ is a $C^{1+\epsilon}$ periodic hypersurface, and $(m_{\ell}, C)$ is a lighthouse measure for some proper cone $C$ and some direction $\ell$ positive on $C$, then $\Sigma=\{x\in\R^n\mid p(\exp(2\pi ix))=0\}$ for some $p$ essentially Lee--Yang polynomial.}

\medskip

The study of Lee--Yang polynomials originated from the celebrated ``Circle Theorem'' of Lee and Yang \cite{LeeYang1952}, where they examine phase transitions in statistical mechanics. These polynomials later became a central object in the extensive work of Borcea and Br\"and\'en \cites{borbra1,borbra2,borbra3} on linear operators preserving stability. Another notable result is the characterization of Lee--Yang polynomials by Ruelle \cite{Ruelle2010}. Our result can be interpreted as a new characterization of Lee--Yang polynomials in terms of their torus zero sets.

The proof of our main theorem is based on the classification of one-dimensional Fourier quasicrystals. A \emph{Fourier quasicrystal} (FQ) is a set\footnote{Possibly with multiplicities.} $\Lambda$ whose counting measure $\mu_{\Lambda}$ is crystalline, and such that $|\widehat{\mu}_{\Lambda}|$ is also tempered. Recent works \cites{kurasovsarnak,olevskii2020fourier,AlonCohenVinzant} classified all one-dimensional Fourier quasicrystals in terms of multivariate {Lee--Yang polynomials}. We now proceed with definitions and a precise statement of the main result \Cref{thm: main}.

\section{Preliminaries and results}
\subsection{Preliminaries} A \emph{cone} $C\subset \mathbb{R}^n$ is a set with the property that $x\in C\iff \lambda x\in C$ for all $\lambda\in (0,\infty)$. We will say that $C$ is a \emph{proper cone}, if it is a convex closed cone with non-empty interior, which is properly contained in a half space: there exists a  direction $\ell\in \mathbb{R}^n$ such that $\langle \ell,x\rangle>0$ for all $x\in C\setminus\{0\}$. The dual cone is defined by 
\[C^*:= \{\ell\in\mathbb{R}^n\mid\forall x\in C:\langle\ell,x\rangle
\ge0\}.\]
If $C$ is proper, then $C^*$ is proper, and $(C^*)^*=C$.
The set of \emph{positive directions} is the interior of the dual cone:
\[\mathrm{int}(C^*):= \{\ell\in\mathbb{R}^n\mid\forall x\in C\setminus\{0\}:\langle\ell,x\rangle>0\}.\]
Recall that for $0<\epsilon<1$ a function is in the class $C^{1+\epsilon}$ if its derivative exists and is $\epsilon$-H\"older continuous. By an \emph{immersed $C^{1+\epsilon}$-hypersurface} in $(\R/\Z)^n$ we mean the image of an immersion $\Sigma=\iota(M)$, where $\iota:M\to(\R/\Z)^n$ is $C^{1+\epsilon}$, and $M$ is a compact $(n-1)$-dimensional manifold. The normal vector field $\nfield$ and the surface area $d\sigma$ on $\Sigma$, are defined on $M$ via local embeddings, and then pushed forward to $\Sigma$ via the immersion $\iota$. Directional measures $m_{\ell}$ are also defined in this way. 
Similarly, a \emph{periodic immersed $C^{1+\epsilon}$-hypersurface} $\Sigma\subset\R^n$ is the preimage of an immersed $C^{1+\epsilon}$-hypersurface in $(\R/\Z)^n$ under the natural projection $\R^n\to(\R/\Z)^n$.
 Locally, up to rotation, such $\Sigma$ is a finite union of graphs of $C^{1+\epsilon}$-functions.  The set of Lee--Yang polynomials $p\in\C[z_1,\ldots,z_n]$ is denoted by $\mathrm{LY}_{n}$.
Any (Laurent) polynomial $p(z_1,\ldots,z_{n})$ defines a periodic analytic variety  
\begin{equation}\label{Sigma(p)}
    \Sigma(p):=\{x\in\R^n\mid p(e^{2\pi i x_{1}},\ldots,e^{2\pi i x_{n}})=0 \}.
\end{equation}
We say that $p$ is \emph{regular} if $p$ and $\nabla p(z)$ do not vanish simultaneously at points $z\in\T^n$ on the product of the unit circle $\T\subset\C$. For regular $p$ the periodic set $\Sigma(p)$ is a smooth manifold.
\begin{Def}[Essentially Lee--Yang polynomial]\label{def: ess LY}
We say that a Laurent polynomial $p$ is \emph{essentially Lee--Yang} with respect to the open cone $\emptyset\neq K \subset \mathbb{R}^n$ if
\[
p(\exp(z)) \neq 0 \quad \text{for all \ }\   z \in \mathbb{C}^n \text{ with}\ \Rea (z)\in\left( K \cup (-K)\right),
\]
where $\exp(z)=(e^{z_1},\ldots,e^{z_n})$ and $\Rea (z)=(\Rea( z_1),\ldots,\Rea( z_n))$.
\end{Def}
\subsection{Main result}
Meyer's observation \cite{Meyer2023}, a proof of which can be found in \cite{alon2025higher}*{Theorem 4.32} when specialized to $d=1$, is that Lee--Yang polynomials provide non-trivial examples of lighthouse measures. By linear change of coordinates, we can state it in greater generality.

First note that if $p$ is essentially Lee-Yang then there exists a proper cone $C\subset\R^n$, such that $p$ is essentially Lee--Yang with respect to the open cone $K=\textnormal{int}(C^*)$ of positive directions of $C$. 
\begin{thm}[Lee--Yang construction of lighthouse]\label{thm: moreover} Let $p$ be a regular essentially Lee--Yang polynomial. Then there exists a proper cone $C\subset\R^n$ (as above) such that, for any  $\ell\in K=\textnormal{int}(C^*)$, the pair $(m_{\ell},C)$ is a lighthouse, where $m_{\ell}$ is the directed measure on $\Sigma(p)$.
\end{thm}

Our main theorem is the inverse result:
\begin{thm}[Lighthouse implies Lee--Yang]\label{thm: main}
    Let $\Sigma\subset\R^n$ be a periodic immersed $C^{1+\epsilon}$-hypersurface. If there exists a proper cone $C$ and a direction $\ell$ that is positive on $C$ and has $\Q$-linearly independent entries, such that $(m_{\ell},C)$ is a lighthouse, then there exists an essentially Lee--Yang polynomial $p$ such that 
\[\Sigma=\Sigma(p)=\{x\in\R^n\mid p(e^{2\pi i x_{1}},\ldots,e^{2\pi i x_{n}})=0 \}.\]
If $\Sigma$ is a manifold, then $p$ is regular.
\end{thm}

\subsection{Characterization of 1D Fourier quasicrystals and breakdown of \Cref{thm: main}}
Our proof is based on a recent characterization of all one-dimensional FQs in terms of Lee--Yang polynomials due to Kurasov and Sarnak \cite{kurasovsarnak}, Olevskii and Ulanovskii \cite{olevskii2020fourier}, and the first author together with Cohen and Vinzant \cite{AlonCohenVinzant}.   
\begin{Def}[Fourier quasicrystal]\label{def:FQ}
    A multiset $\Lambda\subset\R^n$ is called a \emph{Fourier quasicrystal (FQ)} if $\Lambda$ is locally finite and there exists a locally finite set $S(\Lambda)\subset\R$, called the \emph{spectrum of $\Lambda$}, with non-zero complex coefficients $\left(c_{\xi}\right)_{\xi\in S(\Lambda)}$, such that 
    \[\sum_{x\in\Lambda}\hat{f}(x)=\sum_{\xi\in S(\Lambda)}c_{\xi}{f}(\xi),\quad\text{and}\quad \sum_{x\in\Lambda}|{f}(x)|+\sum_{\xi\in S(\Lambda)}|c_{\xi}{f}(\xi)|<\infty,\]
    for every Schwartz function $f$. Here the sums over $\Lambda$ are taken with multiplicities.
\end{Def}
The characterization of one-dimensional Fourier quasicrystals is as follows: 
\begin{thm}[\cites{kurasovsarnak,olevskii2020fourier,AlonCohenVinzant,AlonVinzant}]\label{thm:1dFQ}
    For a multiset $\Lambda\subset\R$ the following are equivalent:
    \begin{enumerate}
        \item $\Lambda$ is a one-dimensional FQ.
        \item $\Lambda$ is the zero set of $f(t)=p(\exp(2\pi it\ell))$ for some $n\in\N$, Lee--Yang polynomial $p\in\mathrm{LY}_{n}$ and positive vector $\ell\in\R^{n}_{+}$ with $\Q$-linearly independent entries. That is,
        \[\Lambda=\{t\in\R\mid t\ell\in\Sigma(p)\},\]
        and the multiplicity of each $t\in\Lambda$ is its multiplicity as a zero of $f$.
    \end{enumerate}
\end{thm}
To prove Theorem \ref{thm: main}, we prove the following two theorems.
\begin{thm}\label{thm: C to FQ}
    Let $\Sigma\subset\R^n$ be a periodic immersed $C^{1+\epsilon}$-hypersurface. Let $C$ be a proper cone and let $\ell\in \mathrm{int}(C^*)$ with $\Q$-linearly independent entries. Consider the measure $m_{\ell}$ on $\Sigma$, and the multiset of intersection ``times''
    \[\Lambda_{\ell}=\{t\in\R\mid t\ell\in\Sigma \},\]
    where the multiplicity of each $t\in \Lambda_\ell$ is the number of sheets of $\Sigma$ that pass through the point $t\ell\in\Sigma$.
    If $(m_{\ell},C)$ is a lighthouse measure, namely 
    \[\mathrm{supp}(\widehat{m}_{\ell})\subset \Z^n\cap\left(C\cup-C\right), \]
    then
    \begin{enumerate}
        \item $\Sigma$ is orientable, it has an orientation such that the normal vector field takes its values in $C\cap S^{n-1}$. 
        \item The following holds for any $f$ Schwartz function on $\R$,
       \begin{equation}\label{eq: summation formula}\sum_{t\in\Lambda_{\ell}}\hat{f}(t)=\sum_{k\in\Z^n}\widehat{m}_{\ell}(k){f}(\langle k,\ell\rangle),
       \end{equation} 
       where the support of the right-hand-side
       $$S(\Lambda)=\{\langle k,\ell \rangle\mid k\in\mathrm{supp}(\widehat{m}_{\ell})\},$$
       is discrete. The multiset $\Lambda_{\ell}$
is a one-dimensional FQ. 
    \end{enumerate}
\end{thm}
\begin{rem}
    The concept of normal vector field is well defined for immersed manifolds, even when the immersion has crossings that result in singularities. If $\Sigma$ is the image of an immersion $\iota\colon M\to\R^n/\Z^n$, by saying that $\Sigma$ has an orientation and normal vector field $\nfield$, we mean that there exists a continuous function $\nfield\colon M\to S^{n-1}$ such that $\nfield(x)$ is in the left-kernel of the differential $D\iota(x)\in\R^{n\times(n-1)}$ for all $x\in M$. Part (1) of Theorem \ref{thm: C to FQ} says that there exists a normal vector field $\nfield$ with the property that $\nfield(x)\in S^{n-1}\cap C$ for all $x\in M$.
\end{rem}
\begin{thm}\label{thm: FQ to p}
    Let $\Sigma\subset\R^n$ be a periodic immersed $C^{1+\epsilon}$-hypersurface. Assume that there exists $\ell\in\R^n$ with $\Q$-linearly independent entries, so that the multiset
\[\Lambda_{\ell}=\{t\in\R\mid t\ell\in\Sigma \}\]
is a one-dimensional FQ with spectrum $S(\Lambda_{\ell})\subset\{\langle\ell,k\rangle\mid k\in\Z^n\}$. Then there exists an essentially Lee--Yang polynomial $p$ such that 
\[\Sigma=\Sigma(p)=\{x\in\R^n\mid p(e^{2\pi i x_{1}},\ldots,e^{2\pi i x_{n}})=0 \}.\] 
If $\Sigma$ is a manifold, then $p$ is regular.
\end{thm}
With Theorem \ref{thm: C to FQ} and Theorem \ref{thm: FQ to p}, this proves Theorem \ref{thm: main}, and also provides a new geometric characterization of one-dimensional FQs:
\begin{cor}
   A uniformly-discrete set $\Lambda\subset\R$ is a one-dimensional FQ (no multiplicities) if and only if there exists
   \begin{enumerate}
       \item a periodic $C^{1+\epsilon}$-hypersurface $\Sigma\subset\R^n$,
       \item a proper cone $C$,
       \item a direction $\ell\in\mathrm{int}(C^{*})$ with $\Q$-linearly independent entries,
   \end{enumerate} 
   such that $(m_{\ell},C)$ is a lighthouse, and $\Lambda$ is the set of intersection times $$\Lambda=\{t\in\R\mid t\ell\in\Sigma\}.$$ 
\end{cor}

\subsection{Summary and structure of the paper}
The central theme of this work is that the lighthouse property for measures imposes rigid geometric structure on the support of the measure. 
We begin by showing that, under mild assumptions, the support of a lighthouse 
measure must be a hypersurface, and that its normal directions are constrained 
to lie inside the lighthouse cone (Theorem~\ref{thm: dimension}). This provides 
the necessary transversality to study intersections with lines.

We explain how the torus zero sets of essentially Lee-Yang polynomials give rise to lighthouse measures. 

The main theorem is the inverse: a hypersurface supporting certain lighthouse measures, up to mild assumptions, must be the torus zero set of an essentially Lee-Yang polynomial. 

\medskip

The paper is organized as follows. In Section~\ref{sec:three} we relate Fourier support to 
normal directions and prove Theorem~\ref{thm: dimension}. In Section~\ref{sec: FQ} we establish 
the connection to one-dimensional Fourier quasicrystals. In Section~\ref{sec:propleeyang} we provide properties of essentially Lee--Yang polynomials, showing the 
equivalent characterizations. In Section~\ref{sec: moreover} we prove \Cref{thm: moreover}, 
showing that essentially Lee--Yang polynomials give rise to lighthouse measures. 
Finally, in Section~\ref{sec: proof of FQ to p} we complete the proof of the inverse theorem, \Cref{thm: main}, using the 
classification of one-dimensional Fourier quasicrystals.

\subsection*{Acknowledgements}
The authors  thank Peter Sarnak, Pavel Kurasov, Wayne Lawton, David Jerison, and Alexander Logunov for insightful discussions. 
The first author was supported by the Department of Mathematics at MIT, and by the Simons Foundation Grant 601948, DJ. The second author was partially supported by DFG grant 502861109.

\section{Proof of Theorem~\ref{thm: dimension}\label{sec:three}, and the relation between normal vectors and the Fourier support of a measure}
In this subsection, we establish a relation between the Fourier support of a measure and the normal directions to its physical support; see Proposition~\ref{prop:normal}. Results closely related to Proposition~\ref{prop:normal} were achieved in the context of wave front sets, see e.g.~\cite{hormander}*{Chapter VIII}, but we did not find its precise statement in the literature.
We then apply this relation to prove Theorem~\ref{thm: dimension}, which asserts that, under mild assumptions, the support of a lighthouse measure has codimension one, it is orientable, and the normal vectors always points into the direction of the lighthouse cone. This will also provide the necessary transversality with respect to directions in $\mathrm{int}(C^{*})$ required in the sequel.
 
\begin{prop}\label{prop:normal}
	Let $m$ be a positive Radon measure on $\R^n$ which is $\Z^n$-periodic. Let $B\subset\R^n$ be an open set and define $$\Sigma=\mathrm{supp}(m)
    \cap B.$$
    If $\Sigma$ is a $C^{1+\epsilon}$-submanifold of $\R^n$, and $m$, when restricted to $B$, has the form
    \[dm|_{B}=\rho d\sigma,\]
    with $\sigma$ the surface measure on $\Sigma$ (induced by the Lebesgue measure on $\R^n$), and $\rho>0$ continuous, then for any $x\in\Sigma$ and any non-zero normal vector $\xi\perp T_{x}\Sigma$,  
\[\left|\{k\in\Z^n\mid \widehat{m}(k)\ne0,\quad \frac{\langle k,\xi\rangle}{\|k\|\|\xi\|}>1-\tau\}\right|=\infty,\quad\text{for all  $\ \tau>0$}.\]
Namely, for any non-zero angle $\theta>0$, the infinite cone $C=C_{\xi,\theta}$ of angle $\theta$ around the direction $\xi$ has
\[ \left|\mathrm{supp}(\widehat{m})\cap C\right|=\infty.\]
\end{prop}	
The proof relies on the fact that for any choice of $N\in\N$ and $\xi$, there is a Gaussian function $g_{N,\xi}$ that is concentrated in physical space in a ball of radius $\frac{\log N}{N}$ around $x_0$, while in Fourier space, it is concentrates in a ball of radius $N\log N$ around $\xi$. The explicit concentration bound we will use is stated below. As this is a standard calculation, we add the proof in the appendix for the interested reader. 
\begin{lem}[Physical and Fourier localization of Gaussians]\label{lem:Gaussian}
	Given the setting and assumptions of Proposition \ref{prop:normal}, we further assume (without loss of generality) that $0\in\Sigma$. For any $N\in \N$ define
    \[C_{N}=\int_{\R^n}e^{-\frac{N^2\|x\|^2}{2}}dm(x).\]
Then, there exists $N_{0}$, such that for any $N>N_{0}$ and any $\xi\in\R^n$,
\begin{align}\label{gaussian 1}
	\frac{1}{C_{N}}\int_{\R^n\setminus B(0,\frac{\log N}{N})}e^{-\frac{N^2\|x\|^2}{2}}dm(x) & \le\ 0.1\\
		\label{gaussian 3}
	\frac{(2\pi)^{n/2}}{C_{N}N^{n}}\sum_{k\in\Z^n\setminus B(\xi,N\log N)}e^{-\frac{\|k-\xi\|^2}{2N^2}} & \le\ 0.1
\end{align}	
\end{lem}
We use the lemma to prove Proposition \ref{prop:normal}.
\begin{proof}[Proof of Proposition \ref{prop:normal}] To ease the notation, we take a smaller $\epsilon>0$, so that $\Sigma$ is $C^{1+2\epsilon}$. Namely, locally, up to rotation, $\Sigma$ is the graph of a function whose derivative is $2\epsilon$-Hölder.
We may normalize $m$ so that $\widehat{m}(0)=\int_{\R^n/\Z^n}dm=1$. We may also assume that $0\in\Sigma$ for simplicity. Fix a non-zero $\xi\perp T_{0}\Sigma$, and for every $N\in \N$ define
\begin{align}
    \xi_{N} & := \ N^{1+\epsilon}\xi
\end{align}
We break the proof of the proposition into two parts.

\textbf{First part:} We claim that there exists a constant $N_{1}$, such that for all $N>N_1$
    \begin{equation}\label{phase estimate}
        \cos(2\pi \langle\xi_{N},x\rangle )\ge 0.9, \quad\text{for all }\  x\in\Sigma\cap B(0,\frac{\log N}{N}).
    \end{equation}
    Let $c=\dim(\Sigma)$ and $d=n-c$.
Since this is a local statement, up to rotation, we may assume that $T_{0}\Sigma=\R^{c}\times\{0\}$ in the $\R^n=\R^{c}\times\R^d$ decomposition. Write $\xi=(0,\omega)$. For large enough $N$, $\Sigma\cap B(0,\frac{\log N}{N})$ {agrees locally with} the graph of a differentiable function $\R^c\colon U\to\R^d$, with $v(0)=0$ and $Dv(0)=0$. In particular, \[\Sigma\cap B(0,\frac{\log N}{N})\subset\{(y,v(y))\mid \|y\|\le \frac{\log N}{N} \},\]
By the assumption, $y\mapsto Dv(y)$ is $2\epsilon$-Hölder continuous, so there exists a constant $C>0$ such that $\|v(y)\|<C \|y\|^{1+2\epsilon}$ for all $y\in\R^c$ {with $\|y\|\le \frac{\log N}{N}$}. Using that $\xi=(0,\omega)$ and $x=(y,v(y))$, we get $\langle\xi,x\rangle=\langle\omega,v(y)\rangle$, so we can bound $$|\langle\xi,x\rangle|\le\|\omega\|\|v(y)\|< C\|\omega\|\|y\|^{1+2\epsilon}. $$  
    In particular,      $\left|\langle\xi_{N},x\rangle \right|\le C\|\omega\|\frac{(\log N)^{1+2\epsilon}}{N^{\epsilon}} \to 0$,
    uniformly over all $x\in \Sigma\cap B(0,\frac{\log N}{N})$. As a result, $\cos(2\pi \langle\xi_{N},x\rangle)\to 1,$
    uniformly over all $x\in \Sigma\cap B(0,\frac{\log N}{N})$, and \eqref{phase estimate} follows.

    \textbf{Second part:} To prove Proposition \ref{prop:normal} it is enough to prove the following claim:
    There exists $N_{1}$ such that for $N>N_{1}$, there exists a point $$k_{N}\in \Z^n\cap B(\xi_{N},N\log N)$$ such that  $\widehat{m}(k_{N})\ne 0$. 
    
    To see why Proposition \ref{prop:normal} follows, notice that on the one hand $\|k_{N}\|\to \infty$, and on the other hand, the angle between such $k_{N}$ and $\xi$ is at most $\theta_{N}={\arcsin}\frac{N\log N}{\|\xi_{N}\|}$, which goes to zero since $\frac{N\log N}{\|\xi_{N}\|}=\frac{N\log N}{N^{1+\epsilon}\|\xi\|}\to 0$. This means that for any $\theta>0$, the infinite cone of angle $\theta$ around $\xi$ will contain every such $k_{N}$ for which $\theta_{N}<\theta$. This will give infinitely many points in the intersection of the cone and the support of $\widehat{m}$.

To prove this claim, and thus prove the proposition, we will upper and lower bound the integral
\begin{align*}
    I_N & :=\frac{1}{C_{N}}\int_{\R^n}e^{-\frac{N^2\|x\|^2}{2}}e^{-2\pi i \langle\xi_{N},x\rangle}dm(x)\\
    & =\frac{(2\pi)^{n/2}}{C_{N}N^{n}}\sum_{k\in\Z^n}e^{-\frac{\|k-\xi_{N}\|^2}{2N^2}}\widehat{m}(k),
\end{align*}
where the last equality is using the definition of $\widehat{m}$, and the Fourier transform of a Gaussian. Now, we can lower bound $|I_{N}|$ 
    	\begin{align*}
		|I_{N}| &  
        \ge \frac{1}{C_{N}}\left|\int_{\R^n\cap B(0,\frac{\log N}{N})}e^{-\frac{N^2\|x\|^2}{2}}e^{2\pi i \langle\xi_{N},x\rangle}dm(x)\right|\\&-\frac{1}{C_{N}}\left|\int_{\R^n\setminus B(0,\frac{\log N}{N})}e^{-\frac{N^2\|x\|^2}{2}}e^{2\pi i \langle\xi_{N},x\rangle}dm(x)\right|  \\
 &       \ge \frac{1}{C_{N}}\int_{\Sigma\cap B(0,\frac{\log N}{N})}e^{-\frac{N^2\|x\|^2}{2}}\cos(2\pi \langle\xi_{N},x\rangle)dm(x)\\ 
 &-\frac{1}{C_{N}}\int_{\R^n\setminus B(0,\frac{\log N}{N})}e^{-\frac{N^2\|x\|^2}{2}}dm(x)  \\
 &\ge \frac{0.9}{C_{N}}\int_{\Sigma\cap B(0,\frac{\log N}{N})}e^{-\frac{N^2\|x\|^2}{2}}dm(x)-0.1   \\
 &\ge (0.9)^2 -0.1 \ge 0.7  
	\end{align*}
    Where the first inequality is simply $|a+b|\ge|a|-|b|$. The second inequality is $|a|\geq\Rea(a)$, and $|\int gdm|\le \int |g|dm$ since $m$ is positive. The third inequality is substituting \eqref{gaussian 1} and \eqref{phase estimate}. The last inequality is \eqref{gaussian 1} again.

On the other hand, $|\widehat{m}(k)|\le \widehat{m}(0)=1$ for all $k$, since $m$ is positive, so we can upper bound $|I_N|$.
\begin{align*}
    |I_{N}| & 
    \le\frac{(2\pi)^{n/2}}{C_{N}N^{n}}\sum_{k\in\Z^n}\left|e^{-\frac{\|k-\xi_{N}\|^2}{2N^2}}\widehat{m}(k)\right|\\
    &
    \le\frac{(2\pi)^{n/2}}{C_{N}N^{n}}\sum_{k\in\Z^n}e^{-\frac{\|k-\xi_{N}\|^2}{2N^2}}\\
    &\le  0.1+\frac{(2\pi)^{n/2}}{C_{N}N^{n}}\sum_{k\in\mathrm{supp}(\widehat{m})\cap B(\xi_{N},N\log N)}e^{-\frac{\|k-\xi_{N}\|^2}{2N^2}},
\end{align*}
 where we used \eqref{gaussian 3} in the last inequality. Comparing the upper bound on $|I_{N}|$ to the lower bound $|I_{N}|\ge 0.7$, we conclude that
\[\sum_{k\in\mathrm{supp}(\widehat{m})\cap B(\xi_N,N\log N)}e^{-\frac{\|k-\xi_{N}\|^2}{2N^2}}\ge 0.6\frac{C_{N}N^{n}}{(2\pi)^{n/2}}>0,\]
which means, in particular, that $\mathrm{supp}(\widehat{m})\cap B(N^{1+\epsilon}\xi,N\log N)\ne\emptyset$. As discussed above, this concludes the proof of the proposition.
\end{proof}
We may now prove Theorem \ref{thm: dimension}.
\begin{proof}[Proof of Theorem \ref{thm: dimension}]
Let $(\mu,C)$ be a $\Z^n$-periodic lighthouse measure in $\R^n$, let $M\subset\R^n$ be a $c$-dimensional $C^{1+\epsilon}$ submanifold with $c$-dimensional volume measure $\sigma$. Let $x_{0}\in\mathrm{supp}(\mu)\cap M$, and assume that for some ball $B$ around $x_{0}$, 
\[d\mu|_{B}=\rho d\sigma|_{B},\]
with a continuous and positive density function $\rho$. 
We first prove the following:\\
\textbf{Claim:} any non-zero $\xi\in \R^n$ with $\xi\perp T_{x_0}M$ must satisfy $\xi\in C\cup -C$. 

To see why, assume that there exists non-zero $\xi\in \R^n$ with $\xi\perp T_{x_0}M$ but such that $\xi\notin (C\cup-C)$, and consider the line $V=\mathrm{span}_{\R}(\xi)$. Since $C$ is a cone, this implies $V\cap (C\cup-C)=\{0\}$. Using the fact that both $V$ and $C\cup-C$ are closed and scaling invariant, there is a minimal angle $\theta$ between $V$ and $C$
\[\theta=\min_{v\in V\setminus\{0\},w\in C\setminus\{0\}}\arccos^{-1}\left(\frac{|\langle v,w\rangle |}{\|v\|\|w\|}\right)>0.\]
On the other hand, Proposition \ref{prop:normal} guarantees that $\widehat{\mu}(k)\ne0$ for some $k\in\Z^n\setminus\{0\}$ in the $\theta/2$ cone around $\xi$. This means that $k\notin C$ and $\widehat{\mu}(k)\ne0$, which is a contradiction to the lighthouse property. Hence, we proved the claim.

Now we prove parts (1) and (2) of the theorem:\\
\textbf{Part (1):} To prove that $c=n-1$, let us assume that $c\ne n-1$ to obtain a contradiction.

By the definition of lighthouse, $\mathrm{Vol}_{n}(M\cap B)=0$, so $c\le n-1$, and since we assume $c\ne n-1$, we must have $c\le n-2$.

By the lighthouse property, the cone $C$ is proper, so there is an $n-1$ dimensional vector subspace $H\subset\R^n$ such that $H\cap (C\cup -C)=\{0\}$. The tangent space $T_{x_0}(M)$ is $c$-dimensional, hence its orthogonal complement $(T_{x_0}(M))^{\perp}$ is $n-c\ge 2$ dimensional, so it must intersect any hyperplane, such as $H$, non-trivially. In particular, there is a non-zero $\xi \in H\cap (T_{x_0}(M))^{\perp}$. However, by the claim proved above, $\xi\in (C\cup -C)$ and $H\cap (C\cup -C)=\{0\}$, giving the needed contradiction.

\textbf{Part (2):} Part (1) and the claim proved above guarantee that $M$ is a manifold of codimension $1$, and for any $x\in M\cap B$, letting $V_x=(T_{x}M)^{\perp}$, we know that $V_{x}\subset C\cup-C$. Let $\nfield(x)=\xi$ for the unique point $\xi\in V_{x}\cap C$ with $\|\xi\|=1$. The map $x\mapsto \nfield(x)$ is locally defined and continuous because $M$ is differentiable, so the only obstruction to defining it globally on $M\cap B$ is if there is a sequence $x_j\in M\cap B$ with $x_{j}\to x$ but $\nfield(x_j)\to -\nfield(x)$. However, the sequence $\nfield(x_j)$ lies in the closed cap $S^{n-1}\cap C$, while $-\nfield(x)$ lies in the reciprocal closed cap $S^{n-1}\cap -C$, and these are disjoint caps as $C$ is proper. Thus, this cannot happen, so there is no obstruction, and hence $M\cap B$ is orientable with orientation such that $\nfield(x)\in C$ for all $x$.
\end{proof}

\section{One-dimensional FQs and a proof of Theorem \ref{thm: C to FQ}}\label{sec: FQ}
Kurasov and Sarnak \cite{kurasovsarnak} showed that there exist one-dimensional FQs that are uniformly discrete and non-periodic, answering several open questions in the field. To do so, Kurasov and Sarnak provided a general construction of one-dimensional FQs, proving $(2)\Rightarrow(1)$ in Theorem \ref{thm:1dFQ}. Olevskii and Ulanovskii proved that every one-dimensional FQ is the zero set of a real-rooted trigonometric function \cite{olevskii2020fourier}. Cohen, Vinzant, and the first author showed that, up to a non-vanishing factor $e^{i\lambda t}$, any real rooted trigonometric function has the form $f(t)=p(\exp(2\pi it\ell))$ for some $n\in\N$, Lee--Yang polynomial $p\in\mathrm{LY}_{n}$ and positive vector $\ell\in\R^{n}_{+}$ with $\Q$-linearly independent entries. This concludes $(1)\Rightarrow(2)$ in Theorem \ref{thm:1dFQ}. 
Surprisingly, one-dimensional FQs and their characterization turned out to be intimately related to lighthouse measures on hypersurfaces. This relation is manifested in \Cref{thm: C to FQ}, which provides a general mechanism for creating summation formulas from intersection points of an irrational line with a periodic hypersurface.

This section is a simplified version of the proof of the summation formula in \cite{alon2025higher} for the case of $d=1$, but extended to immersed manifolds. For the next lemma, recall that for a proper cone $C\subset \R^n$ the interior of its dual cone is 
\[\mathrm{int}(C^*)=\{\ell\in\R^{n}\mid \langle \ell,x\rangle >0\ \text{  for all non-zero  $x\in C$}\}.\]
\begin{lem}\label{lem: bounded}
    Let $C$ be a proper cone and let $\ell\in\mathrm{int}(C^*)$. Then the set 
    $$S=\{\langle\ell,k\rangle\mid k\in\Z^n\cap(C\cup-C)\}$$ is locally finite and 
    $\sum_{k\in\Z^n\cap(C\cup-C)\}}|f(\langle\ell,k\rangle)|<\infty$ for any Schwartz functions $f$.
\end{lem}
\begin{proof}
Consider the set $$S_{R}=\{\langle\ell,k\rangle\mid k\in\Z^n\cap(C\cup-C)\textrm{ and }|\langle\ell,k\rangle|<R\},$$ and let $C_{R}=\{x\in C\mid \langle \ell,x\rangle\le R\}$. Since $C$ is closed and convex cone, $C_{R}$ is also convex and $0\in C_R$. We claim that $C_{R}$ is bounded. To see why, assume it is not bounded to get a contradiction. Then $C_R$ contains an infinite ray, i.e.~there exists $x\in C_{R}$ such that $tx\in C_{R}\setminus\{0\}$ for all $t>0$. However $\langle tx,\ell\rangle\le R$ for all $t>0$ implies that $\langle x,\ell\rangle=0$ in contradiction to $\ell\in\mathrm{int}(C^*)$. So $C_{R}$ is bounded. Since $C_{R}$ scales linearly with $R$, it follows that $\mathrm{vol}(C_{R})=O(R^n)$ and we can conclude that $|C_{R}\cap\Z^n|=O(R^n)$. This implies both that $S_R$ is locally finite and also finiteness of the infinite sum.
\end{proof}
With this, we can finish the proof of Theorem \ref{thm: C to FQ}. As a reminder, the setting of the theorem is that $\Sigma\subset\R^n$ is a $\Z^n$-periodic $C^{1+\epsilon}$-immersed hypersurface, $C\subset \R^n$ a proper cone, $\ell\in\mathrm{int}(C^{*})$ with $\Q$-linearly independent entries, and we assume that the measure $m_{\ell}$ on $\Sigma$ is a lighthouse with cone $C$. Let us further denote the $n-1$ dimensional manifold by $M$ and the $C^{1+\epsilon}$-immersion whose image is $\Sigma$ by $\iota\colon M\to\R^n$. Theorem \ref{thm: C to FQ} states that:
\begin{enumerate}
    \item There is a (Hölder) continuous choice of a normal vector field $\nfield:M\to S^{n-1} $ whose image lies in $S^{n-1}\cap C$.
    \item The multiset $\Lambda=\{t\in\R\mid t\ell\in\Sigma\}$ is a one dimensional FQ. For any Schwartz function $f$ on $\R$,
    \[\sum_{t\in\Lambda}\hat{f}(t)=\sum_{k\in\Z^n}\widehat{m}_{\ell}(k){f}(\langle k,\ell\rangle).\]  
\end{enumerate}
\begin{proof}[Proof of Theorem \ref{thm: C to FQ}]
We work in the setting above. 
Let $\Sigma_{\mathrm{reg}}$ be the regular set of $\Sigma$, i.e.~the points at which $\Sigma$ is locally a manifold, and $\Sigma_{\mathrm{sing}}=\Sigma\setminus\Sigma_{\mathrm{reg}}$ be the singular part. Notice that $\iota^{-1}(\Sigma_{\mathrm{reg}})$ is open and dense in $M$. 
Also, by definition, the differential $D\iota(y)$ at $y\in M$ has a one dimensional left kernel, and depends continuously on $y$, so the normal vector field is locally defined up to a choice of sign. In particular, $y\mapsto |\langle \nfield(y),\ell\rangle|$ is a well defined continuous map. Now, consider $y\in M$ with  $x=\iota(y)\in\Sigma_{\mathrm{reg}}$, so $dm_{\ell}=\rho d\sigma$ around $x$ with  $\rho(x)=|\langle\nfield(y),\ell\rangle|$, and $d\sigma$ encodes the multiplicity in it.\footnote{Notice that a regular point  $x\in\Sigma_{\mathrm{reg}}$ can have several preimage points $y$ with $x=\iota(y)$, if the normal vector field agree on the neighborhoods of all preimages.} 
Theorem \ref{thm: dimension} applied to $m_{\ell}$ says that either:\\
(a) $|\langle\nfield(y),\ell\rangle|=0$, or\\
(b) $\nfield$ is defined around $y$ and takes values in $C\cap S^{n-1}$, so $|\langle\nfield(y),\ell\rangle|\ge \tau$.

By continuity of $|\langle\nfield(y),\ell\rangle|$, we conclude that $M$ is the disjoint union of two open-and-closed submanifolds $M_a$ and $M_{b}$. $M_a$ is the set of points $y$ for which $(a)$ holds: $|\langle\nfield(y),\ell\rangle|=0$.  $M_b$ is the set of points $y$ for which $(b)$ holds. In particular, on $M_{b}$, the condition $|\langle\nfield(y),\ell\rangle|\ge \tau$ holds, which means that the locally defined $\nfield$, with sign such that $\langle\nfield(y),\ell\rangle\ge \tau$, extends to globally defined $\nfield$, that takes values in $C\cap S^{n-1}$. 

We claim that $M_a$ is empty. If there was a point $y\in M_a$, then the immersion of $M_a$ must contain the entire line $V=\{\iota(y)+t\ell\mid t\in\R\}$ since $\ell\in T_{x}\Sigma$ for all regular points in the immersion of $M_a$. In particular, $V\subset \Sigma$. However, $\ell$ has $\Q$-independent entries, which means that $V/\Z^n$ is dense, while $\Sigma/\Z^n$ is a codimension one closed set. We conclude that $M_a$ must be empty, and therefore $M=M_b$ has a well defined normal vector field taking values only in $C$. This proves part (1).

For the second part, notice that now $dm_{\ell}(x)=\langle\nfield(x),\ell\rangle d\sigma$ is positive everywhere (no need for absolute value). By assumption, it further satisfies $\mathrm{supp}(\widehat{m}_{\ell})\subset \Z^n\cap\left(C\cup-C\right)$, and by Lemma \ref{lem: bounded}, the set $\{\langle\ell,k\rangle\mid k\in\mathrm{supp}(\widehat{m}_{\ell})\}$ is discrete, and 
\[\sum_{k\in\mathrm{supp}(\widehat{m}_{\ell})}|\widehat{m}_{\ell}(k)f(\langle\ell,k\rangle)|<\infty,\]
for any Schwartz function $f$ on $\R$. To show that the multiset $\Lambda$ is a one-dimensional FQ we are left with showing that its counting measure is tempered and proving the summation formula \eqref{eq: summation formula}.

To prove the summation formula \eqref{eq: summation formula}, recall that the multiplicity of a point $x\in\Sigma$ is $|\iota^{-1}(x)|$, and it is bounded by some $C'>0$ due to compactness of $\Sigma/\Z^n$. 
Let us denote the multiplicity of $x$ by $\deg_{\Sigma}(x)$ and for a point in $\Lambda=\{t\in\R\mid t\ell\in\Sigma\}$ we can define the multiplicity $\deg_{\Lambda}(t):=\deg_{\Sigma}(t\ell)$. We note that $\sup_{x\in\Sigma}\sum_{t\in[0,1]}\deg_{\Sigma}(x+t\ell)=C''<\infty$ as otherwise, due to compactness of $\Sigma/\Z^n$, we would get a contradiction to the lower bound $\langle\nfield,\ell\rangle\ge\tau>0$. 
Define the positive Radon measure measure on $\R$,
\[\mu_{\Lambda}:=\sum_{t\in\Lambda}\deg_{\Lambda}(t)\delta_{t}.\]
Notice that $\sup_{t\in\R}\mu_{\Lambda}([t,t+1])\le C''<\infty$, so $\mu_{\Lambda}([-R,R]=O(R)$. We conclude that $\mu_{\Lambda}$ is tempered, since it is positive and has linear growth.

Our goal is to prove 
\[\hat{\mu}_{\Lambda}=\sum_{k\in\Z^n}\widehat{m}_{\ell}(k)\delta_{\langle k,\ell\rangle}.\]
Consider a family of univariate functions $\rho_{\epsilon}\in C_{c}^{\infty}(\R)$, such that $\mathrm{supp}(\rho_{\epsilon})\subset (-\frac{\epsilon}{2},\frac{\epsilon}{2})$, $\rho_{\epsilon}\ge 0$, and $\int_{t=-\epsilon/2}^{\epsilon/2}\rho_{\epsilon}(t)dt=1$. In particular, the distribution $\rho_{\epsilon}(t)dt$ converges in the weak topology to $\delta_{0}$, and $\hat{\rho}_{\epsilon}(\omega)\to 1$ for all $\omega\in\R$. Define the function 
\[G_{\epsilon}(t)=\sum_{t'\in\Lambda}\rho_{\epsilon}(t-t')=(\mu_{\Lambda}*\rho_{\epsilon})(t).\]
It follows that 
\[\int_{\R}G_{\epsilon}(t)\hat{f}(t)dt\to\int \hat{f}d\mu_{\Lambda}\quad\text{for any $f$ Schwartz on $\R$}.\]  
To prove the summation formula, and therefore the theorem, it is enough to show 
\[\int_{\R}\widehat{G}_{\epsilon}(t){f}(t)dt\to \sum_{k\in\Z^n}\widehat{m}_{\ell}(k){f}(\langle k,\ell\rangle),\quad\text{for any Schwartz function $f$ on $\R$}.\]
Now define the function $I_\epsilon\colon M\times[-\epsilon/2,\epsilon/2]\to\R^n$ by $I_\epsilon(y,t):=\iota(y)+t\ell$. The transversality condition $\langle\nfield,\ell\rangle\ge \tau> 0$ holds everywhere, due to part (1), which allows to conclude that, for $\epsilon$ small enough, $|I_\epsilon^{-1}(x)|$ is upper bounded by the maximum of $\deg_{\Sigma}$ (which is finite). We can also note that for any $x\in\R^n$,
\[\lim_{\epsilon\to 0}|I_\epsilon^{-1}(x)|=\begin{cases}
    0 &  x\notin\Sigma\\
    \deg_{\Sigma}(x) & x\in\Sigma
\end{cases}.\]
Define the function 
\[F_{\epsilon}\colon\R^n\to\R,\quad F_{\epsilon}(x):=\sum_{(y,t)\in I_{\epsilon}^{-1}(x)}\rho_{\epsilon}(t).\]
We claim that
\begin{enumerate}
    
    \item $G_\epsilon(t)=F_{\epsilon}(t\ell)$ for all $t\in\R$, and
    \item $F_{\epsilon}$ is $\Z^n$-periodic, and has absolutely convergent Fourier series. 
\end{enumerate}
Indeed, (1) follows from the the definitions of the functions. For (2), periodicity is due to the $\Z^n$-periodicity of $I_{\epsilon}$. We are left with computing $\widehat{F_{\epsilon}}(k)$:
\begin{align*}
    \widehat{F_{\epsilon}}(k)& = \int_{\R^n/\Z^n}F_{\epsilon}(x)e^{-2\pi i \langle k,x\rangle}dx\\
    & =\int_{y\in M}\int_{t=-\epsilon/2}^{\epsilon/2}e^{-2\pi i \langle k,\iota(y)+t\ell\rangle} \rho_{\epsilon}(t)\langle \nfield(y),\ell\rangle d\sigma_{M}(y)dt\\
    & =\int_{x\in \Sigma}e^{-2\pi i \langle k,x\rangle}\left(\int_{t\in\R}e^{-2\pi i t \langle k,\ell\rangle} \rho_{\epsilon}(t)dt\right)dm_{\ell}(x)\\
    & = 
    \hat{\rho}_{\epsilon}(\langle k,\ell\rangle)\widehat{m}_{\ell}(k),
\end{align*}
where the second equality is implied by the co-area formula applied to
$(y,t)\mapsto \iota(y)+t\ell$, and $d\sigma_{M}$ is the surface measure on $M$
induced by the immersion. 
Since $m$ is positive and lighthouse, one can bound
\[\sum_{k\in\Z^n}|\widehat{F_{\epsilon}}(k)|\le \widehat{m}_{\ell}(0)\sum_{k\in\Z^n\cap\left(C\cup-C\right)}|\hat{\rho}_{\epsilon}(\langle k,\ell\rangle)|<\infty ,\]
by Lemma \ref{lem: bounded}. This proves (2).

Claims (1) and (2) allow to conclude that $G_{\epsilon}$ has a Fourier expansion
\[G_{\epsilon}(t)=F_{\epsilon}(t\ell)=\sum_{k\in\Z^n}\widehat{F_{\epsilon}}(k)e^{2\pi i t\langle k,\ell\rangle }.\]
Therefore, for any $f$ Schwartz function on $\R$ we have
\[\int_{\R}\widehat{G}_{\epsilon}(t){f}(t)dt=\sum_{k\in\Z^n}\widehat{F_{\epsilon}}(k){f}(\langle k , \ell\rangle).\]

Now, since $\hat{\rho}_{\epsilon}(\langle k,\ell\rangle)\to 1$ for all $k$, for any Schwartz function $f$ on $\R$ we have:
\[\int_{\R}\widehat{G}_{\epsilon}(t){f}(t)dt=\sum_{k\in\Z^n}\hat{\rho}_{\epsilon}(\langle k,\ell\rangle)\widehat{m}_{\ell}(k){f}(\langle k , \ell\rangle)\to \sum_{k\in\Z^n}\widehat{m}_{\ell}(k){f}(\langle k , \ell\rangle).\qedhere\]
\end{proof}
\section{Preliminaries on essentially Lee--Yang polynomials}\label{sec:propleeyang}
The standard definition of a \emph{Lee--Yang polynomial}, see for example \cite{Ruelle2010}, is a polynomial  $p\in\C[z_1,\ldots,z_n]$ such that $p(z)\ne 0 $ for $z\in\mathbb{D}^n\cup(\C\setminus\overline{\mathbb{D}})^n$. Where $\mathbb{D}\subset\C$ is the open unit disc. Equivalently, see \cite{AlonCohenVinzant}*{Lemma 2.4} for example, $p\in\C[z_1,\ldots,z_n]$ is Lee--Yang if and only if $p$ is satisfying
\begin{enumerate}
    \item[(i)] $p(0)\ne 0$, and
    \item[(ii)] $p(\exp(z))\neq0$ for all $z\in\C^n$ with $\Rea(z)\in \R_{+}^n\cup\R_{-}^n$. 
\end{enumerate}
This is the reason for our definition of essentially Lee--Yang Laurent polynomials, as discussed in the introduction.
\begin{Def}
    A Laurent polynomial $p\in\C[z_1^{\pm},\ldots,z_n^{\pm}]$ is \emph{essentially Lee--Yang} with respect to a nonempty open  cone $K$ if $p(\exp(z))\neq0$ for all $z\in\C^n$ with $\Rea(z)\in -K\cup K$.
\end{Def}
\begin{rem}\label{rem:noconstantterm}
    Let $p\in\C[z_1,\ldots,z_n]$ be a polynomial that is not divisible by any non-constant monomial. If $p(\exp(z))\neq0$ for all $z\in\C^n$ with $\Rea(z)\in \R_{+}^n\cup\R_{-}^n$, then $p(0)\neq0$. Indeed, by \cite{shafarevich2}*{Lemma 1 on page 124} the zero set of $p$ in $(\C^*)^n$ is dense in its zero set in $\C^n$ (with respect to the Euclidean topology). Now if $p(0)=0$, then there would be a sequence $(x_i)_{i\in\N}$ in $(\C^*)^n$ with $p(x_i)=0$ and $\lim_{i\to\infty}z_i=0$. In particular, for large enough $i$, there is $z\in\C$ with $x_i=\exp(z)$ and $\Rea(z)\in\R_-^n$, contradicting our assumption. Thus, in the above definition of Lee--Yang polynomial, one can replace (i) by the seemingly weaker condition
    \begin{enumerate}
        \item[(i$'$)] $p$ is not divisible by any non-constant monomial.
    \end{enumerate}
    Since condition (i$'$) can always be achieved by multiplication with a suitable monomial, we will sometimes say, by slight abuse of notation, that a Laurent polynomial $q\in \C[z_1^{\pm},\ldots,z_n^{\pm}]$ is Lee--Yang if it satisfies condition (ii) above.
\end{rem}
For a matrix $A\in\Z^{m\times n}$ and $z\in\C^n$ we will write $ w=z^{A} \in \mathbb{C}^{m} $ for  
\[z^{A}:=(w_{1},\ldots,w_{m}),\quad w_{j} = \prod_{i=1}^{n} z_{i}^{A_{ji}}\]
We use element-wise exponent $\exp(z)=(e^{z_{1}},\ldots e^{z_{n}})$, so that $(\exp(z))^{A}=\exp(Az)$.
We will give a characterization of essentially Lee--Yang polynomials that explains the name: these are Lee--Yang polynomials up to a monomial change of coordinates. We first need a lemma.
\begin{lem}\label{lem:denselinerealrootedessly}
    Let $p\in\C[z_1^{\pm},\ldots,z_n^{\pm}]$ be a Laurent polynomial and $\ell\in\R^n$ a vector with $\Q$-linearly independent entries such that $p(\exp(2\pi i\ell t))$ has only real zeros. There exists $k\in\Z^n$, a matrix $A\in\SL_n(\Z)$, a Lee--Yang polynomial $q\in\C[z_1,\ldots,z_n]$ and a vector $\tilde\ell\in\R^n$ with positive entries such that $q(z)=z^k\cdot p(z^A)$ and $\ell=A\tilde\ell$.
\end{lem}
\begin{proof}
    After multiplying $p$ by a monomial, we can assume that $p$ is a polynomial.
    We consider the \emph{amoeba} of $p$, i.e.~the set
    \begin{equation*}
        \cA(p)=\{\log|z|:=(\log|z_1|,\ldots,\log|z_n|)\in\R^n\mid z\in(\C^*)^n\textrm{ and } p(z)=0\}.
    \end{equation*}
    By \cite{AlonCohenVinzant}*{Corollary 3.3} the punctured line $\R^*\ell$ is disjoint from $\cA(p)$. Now applying \cite{AlonCohenVinzant}*{Lemma 2.5} to both $\ell$ and $-\ell$ shows that there exists a rational polyhedral convex cone $K\subset\R^n$ that contains $\ell$ and such that $-\textnormal{int}(K)\cup\textnormal{int}(K)$ is disjoint from $\cA(p)$. 
    By \cite{toric}*{Theorem 11.1.9} there exist $a_1,\ldots,a_n\in K\cap\Z^n$, such that the matrix $A$ with columns $a_1,\ldots,a_n$ has $\det(A)=1$, and $\tilde\ell_1,\ldots,\tilde\ell_n\geq0$ such that $\ell=\tilde\ell_1a_1+\cdots+\tilde\ell_na_n$. Since the entries of $\ell$ are $\Q$-linearly independent, the $\tilde\ell_i$ must be strictly positive. Then $\ell=A\tilde\ell$ by construction. Consider the Laurent polynomial $\tilde{q}(z)=p(z^A)$ and let $z\in\C^n$ with $\Rea(z)\in \R_{+}^n\cup\R_{-}^n$ and assume that $\tilde{q}(\exp(z))=0$. Then $p(\exp(Az))=0$ and $$\log|\exp(Az)|=\Rea(Az)=A\Rea(z)\in -\textnormal{int}(K)\cup\textnormal{int}(K)$$ by construction. But this  contradicts $-\textnormal{int}(K)\cup\textnormal{int}(K)$ being disjoint from $\cA(p)$. Hence $\tilde{q}$ satisfies (ii) of the above definition of Lee--Yang polynomials. Finally, let $k\in\Z^n$ such that $q(z)=z^k\tilde{q}(z)$ is a polynomial that is not divisible by a non-constant monomial. Then $q$ is a Lee--Yang polynomial by Remark \ref{rem:noconstantterm}.
\end{proof}
\begin{thm}\label{thm:essleeyangchara}
    Let $p\in\C[z_1^{\pm},\ldots,z_n^{\pm}]$ be a Laurent polynomial. The following are equivalent:
    \begin{enumerate}
        \item $p$ is essentially Lee--Yang.
        \item There exists an open cone $\emptyset\neq K\subset\R^n$ such that $p(\exp(2\pi i\ell t)$ is real-rooted for all $\ell\in K$.
        \item There exists a vector $\ell\in\R^n$ with $\Q$-linearly independent entries such that $p(\exp(2\pi i\ell t))$ has only real zeros.
        \item There exists $k\in\Z^n$, a matrix $A\in\SL_n(\Z)$ and a Lee--Yang polynomial $q\in\C[z_1,\ldots,z_n]$ with $q(z)=z^kp(z^A)$.
    \end{enumerate}
\end{thm}
\begin{proof}
    ``$(1)\Rightarrow(2)$'': By definition there exists a nonempty open cone $K\subset\R^n$ such that $p(\exp(z))\neq0$ for all $z\in\C^n$ with $\Rea(z)\in -K\cup K$. Let $\ell\in K$ and $t_0\in\C$ with non-zero imaginary part $0\neq b\in\R$. Then $\Rea(2\pi i\ell t_0)=-2b\pi \ell\in -K\cup K$. Hence $t$ is not a zero of $p(\exp(2\pi i\ell t)$ which is thus real-rooted.

    ``$(2)\Rightarrow(3)$'': This is clear because every non-empty open set contains a point with $\Q$-linearly independent entries.

    ``$(3)\Rightarrow(4)$'': This is part of Lemma \ref{lem:denselinerealrootedessly}.

    ``$(4)\Rightarrow(1)$'':  Consider the open cone $K=A\R_{+
    }^n$. Let $z\in\C^n$ with $\Rea(z)\in A\R_{+
    }^n\cup A\R_{-
    }^n$. Letting $z'=A^{-1}z\in\C^n$, we have $\Rea(z')=A^{-1}\Rea(z)\in \R_{+
    }^n\cup \R_{-
    }^n$ and
    \begin{equation*}
        p(\exp(z))=p(\exp(Az'))=e^{-\langle  k,z'\rangle}q(\exp(z'))\neq0
    \end{equation*}
    because $q$ is Lee--Yang.
\end{proof}

In the following few lemmas, we prove some essential properties of essentially Lee--Yang Laurent polynomials.
We denote by $\T\subset\C$ the unit circle.
\begin{lem}\label{lem:layers}
    Let $p\in\C[z_1^{\pm},\ldots,z_n^{\pm}]$ be essentially Lee--Yang and let $Y\subset\T^n$ be its zero set in $\T^n$. We assume that $p$ is square-free in the sense that it does not contain multiple irreducible factors. The set of points $y\in Y$ where $\nabla p$ does not vanish is dense in $Y$.
\end{lem}
\begin{proof}
    Let $A\in\SL_n(\Z)$ and $k\in\Z^n$ such that $q=z^k\cdot (p\circ\psi_A)$ is Lee--Yang where $\psi_A\colon\T^n\to\T^n,\, z\mapsto z^A$. The zero set $Z$ of $q$ in $\T^n$ is the preimage of $Y$ under $\psi_A$. By \cite{AlonVinzant}*{Proposition 4.2} the claim is true for $q$ and $Z$. Now our claim follows from the fact that $\psi_A\colon\T^n\to\T^n$ is a diffeomorphism.
\end{proof}
We will also need the following result on the intersection with dense lines.
\begin{lem}\label{lem:linedense2}
    Let $\ell\in\R^n$ be a column vector whose entries are linearly independent over $\Q$ and let $\psi_\ell\colon\R\to\T^n$ be the map $\psi_{\ell}(t)=\exp(2\pi i t\ell)$. Let $Y\subset\T^n$ be the zero set of a Laurent polynomial $p\in\C[z_1^{\pm},\ldots,z_n^{\pm}]$. If $p$ is  essentially Lee--Yang, then $\im(\psi_\ell)\cap Y$ is dense in $Y$.
\end{lem}
\begin{proof}
    Without loss of generality we can assume that $p$ is irreducible.
    We identify the tangent space $T_z\T^n$ of $\T^n$ at $z$ with $\R^n$ via the differential of the covering map
    \begin{equation*}
        \R^n\to\T^n,\,\theta\mapsto\exp(2\pi i\theta).
    \end{equation*}
    If there exists a point $y\in Y$ with $\nabla p(y)\neq0$ such that the tangent space ${T}_yY\subset\R^n$ does not contain the vector $\ell$, then the dense line $\im(\psi_\ell)\subset\T^n$ intersects $Y$ in every neighborhood of $y$ by the inverse function theorem. If the tangent space of $Y$ at every such point contained $\ell$, then $p$ would vanish on a translate of the dense line $\im(\psi_\ell)$ and $p$ would be identically zero. Furthermore, the set of $y\in Y$, where $\ell$ is not orthogonal to the gradient of $p$, is a Zariski open subset. Since it is non-empty and since the points of $Y$, where the gradient of $p$ does not vanish, is dense in $Y$ by Lemma \ref{lem:layers}, this Zariski open set it is dense in $Y$. This shows the claim.
\end{proof}
Finally, we record some results on the gradient of essentially Lee--Yang polynomials. We begin with the corresponding result for Lee--Yang polynomials.
\begin{lem}\label{lem:leeyanglowerbound}
    Let $p\in\C[z_1^{\pm},\ldots,z_n^{\pm}]$ be Lee--Yang  and $\ell\in\R_{+}^n$. Define $F\colon\C^n\to\C$ by $F(z)=p(\exp(2\pi i z))$. There exists $c>0$ such that for all $x\in\R^n$ with $F(x)=0$ and $\nabla F(x)\neq0$:  $$\frac{|\langle\nabla F(x),\ell\rangle|}{\|\nabla F(x)\|}\geq c.$$
\end{lem}
\begin{proof}
    Consider the compact set $S=\{v\in\R_{\geq0}^n\mid \|v\|=1\}$. Because $\langle v,\ell\rangle>0$ for all $v\in S$, there exists $c>0$ such that $\langle v,\ell\rangle\geq c$ for all $v\in S$.
    Let $x\in\R^n$ with $F(x)=0$ and $\nabla F(x)\neq0$. Then by \cite{AlonVinzant}*{Proposition 3.5} all non-zero entries of $\nabla F(x)$ have the same phase. Thus $\frac{|\langle\nabla F(x),\ell\rangle|}{\|\nabla F(x)\|}=\langle v,\ell\rangle$ for some $v\in S$. This shows the claim.
\end{proof}

\begin{lem}\label{lem:essleeyanglowerbound}
    Let $p\in\C[z_1^{\pm},\ldots,z_n^{\pm}]$ be essentially Lee--Yang and $\ell\in\R^n$ as in part (3) of Theorem \ref{thm:essleeyangchara}, i.e., a vector with $\Q$-linearly independent entries such that $p(\exp(2\pi i \ell t)$ has only real zeros. 
    Define $G\colon\C^n\to\C$ by $G(z)=p(\exp(2\pi i z))$.
    There exists $c>0$ such that for all $x\in\R^n$ with $G(x)=0$ and $\nabla G(x)\neq0$:  $$\frac{|\langle\nabla G(x),\ell\rangle|}{\|\nabla G(x)\|}\geq c.$$
\end{lem}
\begin{proof}
    By Lemma \ref{lem:denselinerealrootedessly} there are a matrix $A\in\GL_n(\R)$, a Lee--Yang Laurent polynomial $q\in\C[z_1^{\pm},\ldots,z_n^{\pm}]$ and a vector $\tilde\ell\in\R^n$ with positive and $\Q$-linearly independent entries such that $p(\exp(Az))=q(\exp(z))$ for all $z\in\C^n$ and $A\tilde\ell=\ell$. 
    Define $F\colon\C^n\to\C$ by $F(z)=q(\exp(2\pi i z))$, so that $F(z)=G(Az)$. By Lemma \ref{lem:leeyanglowerbound} there exists $C_0>0$ with
      \begin{align*}
        C_0&\leq\frac{|\langle\nabla F(x),\tilde\ell\rangle|}{\|\nabla F(x)\|}=\frac{|\langle A^t\nabla G(Ax),\tilde\ell\rangle|}{\| A^t\nabla G(Ax)\|}=\frac{|\nabla G(Ax),\ell\rangle|}{\|A^t \nabla G(Ax)\|}\\ &\leq\|A^{-t}\|\frac{|\nabla G(Ax),\ell\rangle|}{\|\nabla G(Ax)\|}  
      \end{align*}
      for all $x\in\R^n$ with $F(x)=0$ and $\nabla F(x)\neq0$. Since $F(x)=0$ and $\nabla F(x)\neq0$ is equivalent to $G(Ax)=0$ and $\nabla G(Ax)\neq0$,  since $A$ is invertible, this shows that we can choose $c=\frac{C_0}{\|A^{-t}\|}$.
\end{proof}
\section{Proof of Theorem \ref{thm: moreover}}\label{sec: moreover}
Recall that if $q$ is any polynomial, then $\Sigma(q)$ is periodic by definition. If $q$ is a regular Lee--Yang polynomial, then $\Sigma(q)$ is an orientable periodic hypersurface. Finally, if we denote the normal vector field by $\nfield_{q}$ and the $\ell$-directed measure on $\Sigma(q)$ by $m_{\ell}^{q}$, then the following hold:
\begin{enumerate}
    \item[(i)] There is an orientation with $\nfield_{q}(x)\in\R_{\ge 0}^n$ for all $x\in\Sigma(q)$, see \cite{AlonVinzant}*{Proposition 3.5}. 
    \item[(ii)] For every $\ell\in \R_{+}^n $ one has
    $\mathrm{supp}(\widehat{m}_{\ell}^{q})\subset\Z_{\ge 0}^n\cup\Z_{\le 0}^n$, see \cite{alon2025higher}*{Theorem 4.32}.
\end{enumerate}
Now assume that $p$ is a regular essentially Lee--Yang polynomial.
Then there exists $A\in\Z^{n\times n}$ with $\det(A)\neq0$ such that $q(z)=p(z^A)$ is Lee--Yang. We choose $K=A\R_+^n$, $C=K^*$ and let $\ell\in K$, i.e.~$\ell=A\tilde\ell$ for some $\tilde\ell\in\R_+^n$.
This implies directly that $p$ is essentially Lee--Yang with respect to $K$.
The linear map $y(x)=Ax$ is a bijection from $\Sigma(q)$ to $\Sigma(p)$, so $\Sigma(p)$ is also a hypersurface and is $\Z^n$-periodic by definition. The tangent spaces are transformed similarly: $$v\in T_{x}\Sigma(q)\iff Av\in T_{y(x)}\Sigma(p).$$ Therefore, the normal $\nfield_{q}$ of $\Sigma(q)$ and the normal $\nfield_{p}$ of $\Sigma(p)$, satisfy
\[\nfield_{p}(y(x))=(A^{T})^{-1}\nfield_{q}(x).\]
Now we can relate the directional measures $m^{p}_{\ell}$ and $m^{q}_{\tilde{\ell}}$ on $\Sigma(p)$ and $\Sigma(q)$, respectively. Namely, by definition we have
\begin{equation*}
    dm_\ell^p(y)=|\langle\nfield_p(y),\ell\rangle|d\sigma(y)=|\det(A)|\cdot|\langle\nfield_q(x),\tilde\ell\rangle|d\sigma(x)=|\det (A)|dm_{\tilde\ell}^q(x).
\end{equation*}
Now we can calculate for $k\in\Z^n$:
\begin{align*}
\widehat{m}_\ell^p(k)=&\ \int_{y\in\Sigma(p)/\Z^n}e^{2\pi i \langle k,y\rangle}dm^p_{\ell}(y)\\
    \propto  &\  \int_{x\in\Sigma(q)/A^{-1}\Z^n}e^{2\pi i \langle k,Ax\rangle}dm^q_{\tilde{\ell}}(x)\\
    \propto &\  \widehat{m}^q_{\tilde{\ell}}(A^Tk).
\end{align*}
Now (ii) and the fact that $\tilde{\ell}\in\R_{+}^n$ implies $\widehat{m}_\ell^p(k)\neq0$ only if $A^Tk\in\Z^n_{\ge0}\cup\Z^n_{\le 0}$. Since for $k\in\Z^n$ the condition $A^Tk\in\Z_{\geq0}^n$ is equivalent to $k\in C=K^*$, this implies the claim.

\section{Proof of Theorem \ref{thm: FQ to p}}\label{sec: proof of FQ to p}
Before proving Theorem \ref{thm: FQ to p}, we need a lemma about the intersection of dense lines and immersed manifolds in the torus.
In the following, let $\T\subset\C$ be the unit circle and identify the tangent space $T_z\T^n$ of $\T^n$ at $z$ with $\R^n$ via the differential of the covering map
    \begin{equation*}
        \R^n\to\T^n,\,\theta\mapsto\exp(2\pi i\theta)=(e^{2\pi i \theta_{1}},\ldots,e^{2\pi i \theta_{n}}).
    \end{equation*}

\begin{lem}\label{lem:linedense1}
    Let $\ell\in\R^n$ be a column vector whose entries are linearly independent over $\Q$ and let $\psi_\ell\colon\R\to\T^n$ be the map $\psi_{\ell}(t)=\exp(2\pi i t\ell)$. Further let $\iota\colon X\to\T^n$ be a $C^1$-immersion of a closed $(n-1)$-dimensional manifold $X$, and $X_0$ the closure of $\im(\psi_\ell)\cap \iota(X)$. Assume that $X_0\neq\iota(X)$. Then there exists $x$ in the closure of the interior of $\iota^{-1}(X_0)$ such that the image of the tangent space ${T}_xX$ in $\R^n$ contains $\ell$.
\end{lem}
\begin{proof}
    We first note that it suffices to prove the claim for connected $X$. Let $X_1\subset X$ be the set of $x\in X$ so that the image of the tangent space ${T}_xX$ in $\R^n$ does not contain the vector $\ell$. Because $X$ is $C^1$, the set $X_1$ is open. Further, for every $x\in X_1$ there is an open neighborhood $U$ of $x$ such that the dense line $\im(\psi_\ell)\subset\T^n$ intersects $\iota(U)$ in a dense subset by the inverse function theorem. 
    This shows that $X_1$ is contained in the interior of $\iota^{-1}(X_0)$.
    Thus if $X_1$ has a boundary point $x$, then we are done. Indeed, since $X_1$ is open, such $x$ is not in $X_1$, so that the image of the tangent space ${T}_xX$ in $\R^n$ contains $\ell$. Because $\iota^{-1}(X_0)$ is closed, it contains its boundary point $x$ of $X_1$.
    
    If $X_1$ does not have a boundary point, then it is open and closed. Since we have assumed $X$ to be connected, this implies that $X_1=X$, which implies $X_0=\iota(X)$, or $X_1=\emptyset$. Hence, by assumption, we have $X_1=\emptyset$.
    This implies that the image of the tangent space ${T}_xX$ in $\R^n$ contains the vector $\ell$ for every $x\in X$. But then $\iota(X)$ contains a translate of the dense line $\im(\psi_\ell)$. As $\iota(X)$ is closed, this contradicts $X$ being of dimension $n-1$.
\end{proof}

The next proposition is very close to the main theorem of the section. We state it separately because it might be of independent interest.

\begin{prop}\label{prop:nosmoothness}
  Let $X\subset\T^n$ be a closed subset and let $\ell\in\R^n$ have $\Q$-linearly independent entries. Assume that the set $\Lambda=\{t\in\R\mid\exp(2\pi i t\ell)\in X\}$ satisfies the following two conditions:
  \begin{enumerate}
   \item There are positive integer weights $n_t\in\N$ such that $\mu=\sum_{t\in\Lambda}n_t\delta_t$ is a Fourier quasicrystal whose spectrum $S(\mu)$ is contained in the subgroup of $\R$ generated by the entries of $\ell$.
   \item The set $\{\exp(2\pi i t\ell)\mid t\in\Lambda\}$ is dense in $X$.
  \end{enumerate}
  Then $X$ is the zero set in $\T^n$ of an essentially Lee--Yang Laurent polynomial $p$ with the property that the exponential polynomial $p(\exp(2\pi i t\ell))$ is real rooted.
\end{prop}
\begin{proof}
    By \cite{olevskii2020fourier}*{Theorem 8(i)} there is a real rooted exponential polynomial
    \begin{equation*}
        f(t)=\sum_{j=0}^Nc_je^{2\pi i \lambda_jt}
    \end{equation*}
    for some $N\in\N$, $c_j\in\C^\times$ and $0=\lambda_0<\lambda_1<\cdots<\lambda_N$ whose zero set is $\Lambda$. 
    It follows from the proof of \cite{olevskii2020fourier}*{Theorem 8(i)} that one can choose the frequencies $\lambda_j$ from the subgroup of $\R$ generated by $S(\mu)$. Then, by assumption, the $\lambda_j$ are in the subgroup of $\R$ generated by the entries of $\ell$. Thus there are (unique) $k_0,\ldots,k_N\in\Z^n$ with $\lambda_j=\langle k_j,\ell\rangle$ for $j=0,\ldots,N$.  We claim that the Laurent polynomial $p=\sum_{j=0}^Nc_jz^{k_j}$ has the desired properties. By construction
    \begin{equation}\label{eq:pessly}
        p(\exp(2\pi i t\ell))=f(t),
    \end{equation}
    which is real rooted. Moreover, this implies that $p$ is essentially Lee--Yang by Theorem \ref{thm:essleeyangchara}. Let $Y\subset\T^n$ be the zero set of $p$. Then:
    \begin{equation*}
        Y=\overline{Y\cap\{\exp(2\pi i t\ell)\in\T^n\mid t\in\R\}}=\overline{\{\exp(2\pi i t\ell)\in\T^n\mid t\in\Lambda\}}=X,
    \end{equation*}
    where the first equality is Lemma \ref{lem:linedense2}, the second one follows from (\ref{eq:pessly}) and the fact that $\Lambda$ is the zero set of $f$, and the last equality follows from assumption (2).
\end{proof}

Now we are ready to prove the main theorem of the section.
\begin{thm}\label{thm:FQimpliesalgebraic}
    Let $\iota\colon X\to\T^n$ be a $C^1$-immersion of a closed $(n-1)$-dimensional manifold $X$. Let $\ell\in\R^n$ have $\Q$-linearly independent entries. If the multiset $\Lambda=\{t\in\R\mid\exp(2\pi i t\ell)\in\iota(X)\}$ is a one-dimensional FQ whose spectrum $S(\Lambda)$ is contained in the subgroup of $\R$ generated by the entries of $\ell$, then $\iota(X)$ is the zero set in $\T^n$ of an essentially Lee--Yang Laurent polynomial $p\in\C[z_1^{\pm},\ldots,z_n^{\pm}]$.  
\end{thm}
\begin{proof}
    Let $Y$ be the closure of $\{\exp(2\pi i t\ell)\mid t\in\Lambda\}$. By Proposition \ref{prop:nosmoothness} it is the zero set of an essentially Lee--Yang Laurent polynomial $p$ such that $p(\exp(2\pi i t\ell))$ is real rooted. Because $\iota(X)$ is closed, we have $Y\subset\iota(X)$. Assume for the sake of a contradiction that this is a strict inclusion.   
    By Lemma \ref{lem:linedense1} there exists an open subset $U\subset X$ with $\iota(U)\subset Y$ and     
    $x\in\overline{U}$ such that the image of the tangent space ${T}_xX$ in $\R^n$ contains $\ell$. 
    By Lemma \ref{lem:layers} there is a sequence $(x_i)_{i\in\N}\subset Y$ that converges to $\iota(x)$ such that $\nabla p(x_i)\neq0$ for all $i\in\N$.    
    Letting $G\colon\C^n\to\C$ defined by $G(z)=p(\exp(2\pi i z))$, this implies in particular that $\lim_{i\to\infty}\frac{|\langle\nabla G(y_i),\ell\rangle|}{\|\nabla G(y_i)\|}=0$ for a sequence $(y_i)_{i\in\N}\subset\R^n$ with $G(y_i)=0$ and $\nabla G(y_i)\neq0$: choose $y_i$ in the preimage of $x_i$ under $\R^n\to\T^n,\,y\mapsto\exp(2\pi iy)$. However, because $p(\exp(2\pi i \ell t))$ has only real zeros, this is a contradiction to Lemma \ref{lem:essleeyanglowerbound}. Hence $Y=\iota(X)$.
\end{proof}

\begin{proof}[Proof of Theorem \ref{thm: FQ to p}]
    We apply Theorem \ref{thm:FQimpliesalgebraic} to
    \begin{equation*}
        X=\{\exp(2\pi i \theta)\mid\theta\in\Sigma\}\subset\T^n
    \end{equation*}
    and get that 
    \[\Sigma=\{x\in\R^n\ \mid \ p(\exp(2\pi i x))=0 \},\]
    for some essentially Lee--Yang polynomial Laurent polynomial $p$. Finally, we can multiply $p$ by a suitable monomial, which does not change its zero set in $\T^n$,  such that the product is a polynomial. 
\end{proof}
\section{Appendix}

\subsection{Gaussian Preliminaries}
To prove Lemma \ref{lem:Gaussian}, we need the following preliminaries:
\begin{enumerate}
    \item[(i)] The integral of an $n$ dimensional Gaussian with variance $N^{-2}$ is  
\begin{equation}\label{eq: gauss int}
    \int_{\R^n}e^{-\tfrac{N^2\|x\|^2}{2}}dx=\tfrac{(2\pi)^{n/2}}{N^n}.
\end{equation}
\item[(ii)] There exists $c_n>0$ such that for all $R>0$ the following inequality holds:
\begin{equation}\label{eq: gauss tail}
 \int_{\R^n\setminus B_{R}(0)}e^{-\tfrac{N^2\|x\|^2}{2}}dx\le
\tfrac{c_{n}}{RN^{n+1}} e^{-\tfrac{N^2R^2}{2n}}.
\end{equation}
\item[(iii)] 
There exists $c_n'>0$ such that for any $N\in\N$, any $R>1$, and for all $\xi\in\R^n$,
\begin{equation}\label{eq: disc gauss tail}
    \sum_{k\in\Z^n\setminus B_{RN}(\xi)}e^{-\tfrac{\|k-\xi\|^2}{2N^2}}\leq c'_n (RN)^ne^{-\frac{R^2}{2}}.
\end{equation}
\end{enumerate}
\begin{proof}[Proof of (ii)]
    For $N=1$, take $n$ i.i.d.~normal random variables $Z_{j}\sim \mathrm{N}(0,1)$. It is a standard estimate that  
$\mathbb{P}(\|(Z_{1},\ldots,Z_{n})\|>t)\le \frac{c_{n}}{t}e^{-\frac{t^2}{2n}}$ for some $c_{n}>0$. To see why, notice that $\|(Z_{1},\ldots,Z_{n})\|>t$ implies $|Z_{j}|>\frac{t}{\sqrt{n}}$ for some $j$, and for a single variable we have $\mathbb{P}(|Z_{j}|>t)\le \frac{1}{\sqrt{2\pi}t}e^{-\frac{t^2}{2}} $, see for example \cite{vershynin2018high}*{Proposition 2.1.2}. Now change of variables proves it for all $N$.
\end{proof}
\begin{proof}[Proof of (iii)]
Consider the lattice $L=\frac{1}{\sqrt{2\pi}N}\cdot\Z^n$ and let $u=-\frac{\xi}{\sqrt{2\pi}N}$. Then we have
\[\sum_{k\in\Z^n\setminus B_{RN}(\xi)}e^{-\tfrac{\|k-\xi\|^2}{2N^2}}=  \sum_{x\in (u+L)\setminus B_{R/\sqrt{2\pi}}(0)}e^{-\pi\|x\|^2}.\]
By \cite{Banaszczyk}*{Lemma 1.5.(ii)} we have
\begin{equation*}
    \sum_{x\in (u+L)\setminus B_{R/\sqrt{2\pi}}(0)}e^{-\pi\|x\|^2} \leq 2  R^n \left(\frac{e}{n}\right)^{\frac{n}{2}}e^{-\frac{R^2}{2}}\cdot \sum_{x\in L}e^{-\pi\|x\|^2}.
\end{equation*}
Next, by \cite{Banaszczyk}*{Lemma 1.4.(i)} we can bound
\begin{equation*}
    \sum_{x\in L}e^{-\pi\|x\|^2}\leq \left(\sqrt{2\pi}N\right)^{{n}}\cdot\sum_{k\in\Z^n}e^{-\pi\|k\|^2}.
\end{equation*}
Combining the above identity proves (iii) with 
\begin{equation*}
     c'_n=2\cdot\left(\frac{2\pi e}{n}\right)^{\frac{n}{2}}\cdot\sum_{k\in\Z^n}e^{-\pi\|k\|^2}.\qedhere
\end{equation*}
\end{proof}
\subsection{Proof of Lemma \ref{lem:Gaussian}}
As a reminder, the setting of the lemma is as follows. We assume $m$ is a positive Radon measure on $\R^n$ that is $\Z^n$-periodic. We assume that there is some $B\subset\R^n$, an open ball around the origin, such that $\Sigma:=\mathrm{supp}(m)\cap B$ is a $C^{1+\epsilon}$ submanifold of dimension {$c$} and $0\in\Sigma$. We further assume that there is a positive continuous function $\rho$ such that 
\[dm|_{B}=\rho d\sigma,\]
where $\sigma$ is the surface measure on $\Sigma$ (induced from the Lebesgue measure on $\R^n$). For every $N\in\N$ we define the normalization constant
    \[C_{N}=\int_{\R^n}e^{-\frac{N^2\|x\|^2}{2}}dm(x).\]
With this setting, Lemma  \ref{lem:Gaussian} states that there exists $N_{0}$, such that for any $N>N_{0}$ and any $\xi\in\R^n$,
\begin{align}
	\frac{1}{C_{N}}\int_{\R^n\setminus B(0,\frac{\log N}{N})}e^{-\frac{N^2\|x\|^2}{2}}dm(x) & \le\ 0.1\\
	\frac{(2\pi)^{n/2}}{C_{N}N^{n}}\sum_{k\in\Z^n\setminus B(\xi,N\log N)}e^{-\frac{\|k-\xi\|^2}{2N^2}} & \le\ 0.1
\end{align}

\begin{proof}[Proof of Lemma \ref{lem:Gaussian}]
We may normalize $m$ such that $\widehat{m}(0)=\int_{\R^n/\Z^n}dm=1$. We now estimate $C_N$. Given $r>0$, we can estimate 
\[C_{N}=C_{N,r}+\mathrm{Err}(N,r),\]
with
\begin{equation*}
    C_{N,
r}=  \int_{\|x\|<r}e^{-\frac{N^2\|x\|^2}{2}}dm(x)\textrm{ and }
\mathrm{Err}(N,r) =  
\int_{\|x\|\ge r}e^{-\frac{N^2\|x\|^2}{2}}dm(x).
\end{equation*}   
Since 
{ $$e^{-\frac{(N+1)^2\|x\|^2}{2}}=e^{-\frac{(2N+1)\|x\|^2}{2}}e^{-\frac{N^2\|x\|^2}{2}}\le e^{-\frac{(2N+1)r^2}{2}}e^{-\frac{N^2\|x\|^2}{2}} $$}
when $\|x\|\ge r$, we can conclude that $\mathrm{Err}(N+1,r)\le e^{-\frac{(2N+1)r^2}{2}}\mathrm{Err}(N,r)$ for all $N$. 
Since $\mathrm{Err}(1,r)$ is finite, iterating the inequality  gives
{\[\mathrm{Err}(N,r)\le e^{-\frac{r^2}{2}\sum_{j=2}^N(2j+1)}\mathrm{Err}(1,r)=O(e^{-\frac{N^2 r^2}{2}}). \]}
By setting $r=\frac{\log N}{N}$, 
and taking $N$ large enough so that $B(0,\frac{\log N}{N})\subset B$ we conclude that 
\[C_{N}=\int_{x\in \Sigma\cap B(0,\frac{\log N}{N})}e^{-\frac{N^2\|x\|^2}{2}}\rho(x)d\sigma(x)+O(N^{-\frac{\log N}{2}}).\]
Let $c=\dim(\Sigma)$ and $d=n-c$.
Up to rotation, we may assume that $T_{0}\Sigma=\R^{c}\times\{0\}$ in the $\R^n=\R^{c}\times\R^d$ decomposition. For large enough $N$, $\Sigma\cap B(0,\frac{\log N}{N})$ is the graph of a $C^{1+\epsilon}$-function $v\colon\R^{c}\to\R^d$, with $v(0)=0$ and $Dv(0)=0$. In particular, \[\Sigma\cap B(0,\frac{\log N}{N})\subset\{(y,v(y))\mid \|y\|\le \frac{\log N}{N} \},\]
and there exists a constant $C>0$ such that $\|v(y)\|<C \|y\|^{1+\epsilon}$ for all $y\in\R^c$ {with $\|y\|\le \frac{\log N}{N}$}. In the coordinates $x=(y,v(y))$ the surface measure is $d\sigma(x)=\sqrt{\det(\mathrm{id+Dv(y)^T Dv(y)})}dy$. The following function is continuous and strictly positive
\[\varphi(y)=\rho(y,v(y))\sqrt{\det(\mathrm{id+Dv(y)^T Dv(y)})},\]
so $\varphi(y)=\varphi(0)(1+o(N))$ uniformly over $y\in B(0,\frac{\log N}{N})$. Therefore,
\[C_{N,\frac{\log N}{N}}=\varphi(0)(1+o(N))\int_{\|y\|^2+\|v(y)\|^2<\frac{\log^2 N}{N^2}}e^{-\frac{N^2(\|y\|^2+\|v(y)\|^2)}{2}}dy.\]
Notice that when $\|y\|<\frac{\log N}{N}$, for large enough $N$, $\|v(y)\|<0.1\|y\|$, in which case
\begin{align*}
\|y\|^2
\le & \|y\|^2+\|v(y)\|^2\le 1.01\|y\|^2\\
  e^{-\frac{(1.01)N^{2}\|y\|^2}{2}}< & e^{-\frac{N^2(\|y\|^2+\|v(y)\|^2)}{2}}<e^{-\frac{N^2\|y\|^2}{2}}  
\end{align*}
Using \eqref{eq: gauss int} and $\eqref{eq: gauss tail}$ we conclude
\begin{align*}
  C_{N,\frac{\log N}{N}}
\le & \ \varphi(0)(1+o(N))\int_{\|y\|^2<\frac{\log^2 N}{N^2}}e^{-\frac{N^2\|y\|^2}{2}}dy\\
= & \ \varphi(0)(1+o(N))\left(\frac{(2\pi)^{\frac{c}{2}}}{N^c}+O(N^{-\frac{\log N}{2n}})\right).  
\end{align*}
and since $\|y\|^2\le \frac{\log^2 N}{1.01N^2}\Rightarrow \|y\|^2+\|v(y)\|^2\le \frac{\log^2 N}{N^2}$, we can lower bound
\begin{align*}
  C_{N,\frac{\log N}{N}}
\ge & \ \varphi(0)(1+o(N))\int_{\|y\|^2<\frac{\log^2 N}{1.01 N^2}}e^{-\frac{1.01 N^2\|y\|^2}{2}}dy\\
= &\ \varphi(0)(1+o(N))(\frac{(2\pi)^{\frac{c}{2}}}{({\sqrt{1.01}}\  N)^c}+O(N^{-\frac{\log N}{2n}})).  
\end{align*}
We conclude that there are $c_1,c_2$ positive constants such that for large enough $N$
\[c_1 N^{-c}\le C_{N,\frac{\log N}{N}}\le c_2 N^{-c}.\]	Now, using $\mathrm{Err}(N,\frac{\log N}{N})=O(N^{-\frac{\log N}{2}})$ we prove \eqref{gaussian 1}  
\[\frac{\mathrm{Err}(N,\frac{\log N}{N})}{C_{N}}\le  \frac{O(N^{-\frac{\log N}{2}})}{c_1 N^{-c}+O(N^{-\frac{\log N}{2}})}=O(N^{{ c}-\frac{\log N}{2}})\le 0.1,\]
for large enough $N$. Using \eqref{eq: disc gauss tail} with $R=\log N$ we can bound, for any $\xi\in\R^n$ 
\begin{equation*}
    \frac{(2\pi)^{n/2}}{C_{N}N^{n}}\sum_{k\in\Z^n\setminus B(\xi,N\log N)}e^{-\tfrac{\|k-\xi\|^2}{2N^2}}\leq \frac{(2\pi)^{n/2}}{c_{1}N^{n-c}} c'_n (N\log N)^n N^{-\frac{\log N}{2}}\le 0.1,
\end{equation*}
for large enough $N$, proving \eqref{gaussian 3}.
\end{proof}

\bibliographystyle{plain}
\bibliography{biblio}

\begin{bibdiv}
\begin{biblist}

\bib{AlonCohenVinzant}{article}{
      author={Alon, Lior},
      author={Cohen, Alex},
      author={Vinzant, Cynthia},
       title={Every real-rooted exponential polynomial is the restriction of a
  {Lee}-{Yang} polynomial},
        date={2024},
        ISSN={0022-1236},
     journal={J. Funct. Anal.},
      volume={286},
      number={2},
       pages={10},
        note={Id/No 110226},
}

\bib{alon2025higher}{article}{
      author={Alon, Lior},
      author={Kummer, Mario},
      author={Kurasov, Pavel},
      author={Vinzant, Cynthia},
       title={Higher dimensional {Fourier} quasicrystals from {Lee}--{Yang}
  varieties},
        date={2025},
        ISSN={0020-9910},
     journal={Invent. Math.},
      volume={239},
      number={1},
       pages={321\ndash 376},
}

\bib{AlonVinzant}{article}{
      author={Alon, Lior},
      author={Vinzant, Cynthia},
       title={Gap distributions of {Fourier} quasicrystals with integer weights
  via {Lee}-{Yang} polynomials},
        date={2024},
        ISSN={0213-2230},
     journal={Rev. Mat. Iberoam.},
      volume={40},
      number={6},
       pages={2203\ndash 2250},
}

\bib{Banaszczyk}{article}{
      author={Banaszczyk, W.},
       title={New bounds in some transference theorems in the geometry of
  numbers},
        date={1993},
        ISSN={0025-5831},
     journal={Math. Ann.},
      volume={296},
      number={4},
       pages={625\ndash 635},
         url={https://eudml.org/doc/165105},
}

\bib{borbra1}{article}{
      author={Borcea, Julius},
      author={Br{\"a}nd{\'e}n, Petter},
       title={The {Lee}--{Yang} and {P{\'o}lya}--{Schur} programs. {I}:
  {Linear} operators preserving stability},
        date={2009},
        ISSN={0020-9910},
     journal={Invent. Math.},
      volume={177},
      number={3},
       pages={541\ndash 569},
}

\bib{borbra2}{article}{
      author={Borcea, Julius},
      author={Br{\"a}nd{\'e}n, Petter},
       title={The {Lee}--{Yang} and {P{\'o}lya}--{Schur} programs. {II}:
  {Theory} of stable polynomials and applications},
        date={2009},
        ISSN={0010-3640},
     journal={Commun. Pure Appl. Math.},
      volume={62},
      number={12},
       pages={1595\ndash 1631},
}

\bib{borbra3}{article}{
      author={Br{\"a}nd{\'e}n, Petter},
       title={The {Lee}-{Yang} and {P{\'o}lya}-{Schur} programs. {III}:
  {Zero}-preservers on {Bargmann}-{Fock} spaces},
        date={2014},
        ISSN={0002-9327},
     journal={Am. J. Math.},
      volume={136},
      number={1},
       pages={241\ndash 253},
  url={muse.jhu.edu/journals/american_journal_of_mathematics/summary/v136/136.1.branden.html},
}

\bib{toric}{book}{
      author={Cox, David~A.},
      author={Little, John~B.},
      author={Schenck, Henry~K.},
       title={Toric varieties},
      series={Grad. Stud. Math.},
   publisher={Providence, RI: American Mathematical Society (AMS)},
        date={2011},
      volume={124},
        ISBN={978-0-8218-4819-7},
}

\bib{hormander}{book}{
      author={H{\"o}rmander, Lars},
       title={The analysis of linear partial differential operators. {I}:
  {Distribution} theory and {Fourier} analysis.},
     edition={Reprint of the 2nd edition 1990},
      series={Class. Math.},
   publisher={Berlin: Springer},
        date={2003},
        ISBN={3-540-00662-1},
}

\bib{kurasovsarnak}{article}{
      author={Kurasov, Pavel},
      author={Sarnak, Peter},
       title={Stable polynomials and crystalline measures},
        date={2020},
        ISSN={0022-2488},
     journal={J. Math. Phys.},
      volume={61},
      number={8},
       pages={083501, 13},
}

\bib{Meyer2023}{article}{
      author={Meyer, Yves},
       title={Multidimensional crystalline measures},
        date={2023},
        ISSN={0368-6310},
     journal={Trans. R. Norw. Soc. Sci. Lett.},
      volume={2023},
      number={1},
       pages={1\ndash 24},
}

\bib{olevskii2020fourier}{article}{
      author={Olevskii, Alexander},
      author={Ulanovskii, Alexander},
       title={Fourier quasicrystals with unit masses},
        date={2020},
        ISSN={1631-073X},
     journal={C. R., Math., Acad. Sci. Paris},
      volume={358},
      number={11-12},
       pages={1207\ndash 1211},
}

\bib{Ruelle2010}{article}{
      author={Ruelle, David},
       title={Characterization of {Lee}-{Yang} polynomials},
        date={2010},
        ISSN={0003-486X},
     journal={Ann. Math. (2)},
      volume={171},
      number={1},
       pages={589\ndash 603},
         url={annals.princeton.edu/annals/2010/171-1/p16.xhtml},
}

\bib{shafarevich2}{book}{
      author={Shafarevich, Igor~R.},
       title={Basic algebraic geometry. 2: {Schemes} amd complex manifolds.
  {Transl}. from the {Russian} by {Miles} {Reid}.},
     edition={2nd, rev. and exp. ed.},
   publisher={Berlin: Springer-Verlag},
        date={1994},
        ISBN={3-540-57554-5},
}

\bib{vershynin2018high}{book}{
      author={Vershynin, Roman},
       title={High-dimensional probability. {An} introduction with applications
  in data science},
      series={Camb. Ser. Stat. Probab. Math.},
   publisher={Cambridge: Cambridge University Press},
        date={2018},
      volume={47},
        ISBN={978-1-108-41519-4; 978-1-108-23159-6},
}

\bib{LeeYang1952}{article}{
      author={Yang, C.~N.},
      author={Lee, T.~D.},
       title={Statistical theory of equations of state and phase transitions.
  {II}. {L}attice gas and {I}sing model},
        date={1952},
     journal={Physical Review},
      volume={87},
      number={3},
       pages={410\ndash 419},
}

\end{biblist}
\end{bibdiv}

 \end{document}